\documentclass{amsart}
\usepackage{amsmath,amsthm,amssymb}
\usepackage[arrow,matrix]{xy}
\usepackage{hyperref}
\usepackage{pst-node}

\theoremstyle{plain}
\newtheorem{thm}{Theorem}[section]
\newtheorem{lem}[thm]{Lemma}
\newtheorem{prop}[thm]{Proposition}

\theoremstyle{definition}
\newtheorem{defn}{Definition}[section]
\newtheorem{prob}{Problem}

\DeclareMathOperator{\type}{type}
\DeclareMathOperator{\cc}{cc}
\def\k{\mathbf{k}}
\def\Z{\mathbb{Z}}
\def\C{\mathcal{C}}
\def\vput(#1)#2{\cnode*(#1){1pt}{#2}}
\def\edge#1#2{\ncline{#1}{#2}}

\begin{document}
\title[Combinatorial rigidity] {Combinatorial rigidity of
  $3$-dimensional simplicial polytopes}

\author[S.Choi]{Suyoung Choi}
\address{Department of Mathematics, Osaka City University, Sugimoto, Sumiyoshi-ku, Osaka 558-8585, Japan}
\email{\href{mailto:choi@sci.osaka-cu.ac.jp}{choi@sci.osaka-cu.ac.jp}}
\urladdr{\url{http://math01.sci.osaka-cu.ac.jp/~choi}}

\author[J.S.Kim]{Jang Soo Kim}
\address{LIAFA, University of Paris 7, 175 rue du Chevaleret, Paris, France}
\email{\href{mailto:jskim@kaist.ac.kr}{jskim@kaist.ac.kr}}
\urladdr{\url{http://liafa.jussieu.fr/~jskim}}

\thanks{The first author was supported by the Japanese Society for the Promotion
  of Sciences (JSPS grant no. P09023).}  
\thanks{The second author is supported
  by the grant ANR08-JCJC-0011.}


\date{\today}
\maketitle

\begin{abstract}
  A simplicial polytope is \emph{combinatorially rigid} if its combinatorial
  structure is determined by its graded Betti numbers which are important
  invariant coming from combinatorial commutative algebra. We find a necessary
  condition to be combinatorially rigid for $3$-dimensional reducible simplicial
  polytopes and provide some rigid reducible simplicial polytopes.
\end{abstract}

\section{Introduction} \label{section:intro}

Davis and Januskiewicz~\cite{Da-Ja-1991}
introduced the notion of what is now
called a \emph{quasitoric manifold} which is a real $2d$-dimensional closed
smooth manifold $M$ with a locally standard smooth action of $T:=(S^1)^d$ whose
orbit space can be identified with a simple polytope. We note that there is a
natural bijection between the set of simple polytopes and the set of simplicial
polytopes via the dual operation. A quasitoric manifold $M$ is said to be
\emph{over} a simplicial polytope $P$ if the orbit space of $M$ can be
identified with the dual of $P$.\footnote{Many toric topologists prefer to use
  the terminology `simple' instead of `simplicial'. However, to simplify the
  arguments, we define all notions in terms of simplicial polytopes throughout
  this paper.} By Davis and Januskiewicz \cite{Da-Ja-1991}, the equivariant
cohomology ring $H^\ast_T(M) = H^\ast(ET \times_T M)$ with $\Z$-coefficient is
isomorphic to the face ring $\Z[P]$ of $P$ as a graded ring, where $ET$ is a
contractible space which admits a free $T$-action. We note that the natural
projection $p:ET \times_T M \to BT$ induces a $H^\ast(BT)$-module structure of
$H^\ast_T(M)$, where $BT:=ET/T$. They also showed that it is a free-module,
i.e., $H^\ast_T(M) \cong H^\ast(M) \otimes H^\ast(BT)$. Hence, we deduce that
$H^\ast(M) \cong H^\ast_T(M)/p^\ast(H^\ast(BT)) \cong \Z[P]/p^\ast(H^\ast(BT))$,
where $p^\ast : H^\ast(BT) \to H^\ast_T(M)$ is the induced map of $p$. See
\cite[Section 5]{Bu-Pa-2002} for more details.  Thus $H^\ast(M)$ contains some
information of the orbit space $P$. With this viewpoint, Choi et
al. \cite{Ch-Pa-Su-2010} defined the cohomological rigidity of $P$ as
follows.\footnote{The original definition of cohomological rigidity was firstly
  introduced by Masuda and Suh in \cite{Ma-Su-2008} in terms of toric manifolds
  and fans.} A simplicial polytope $P$ is \emph{cohomologically rigid} if there
exists a quasitoric manifold $M$ over $P$, and whenever there exists a
quasitoric manifold $N$ over another polytope $Q$ with a graded ring isomorphism
$H^*(M)\cong H^*(N)$, then $P=Q$ up to isomorphism. Choi et al.~
\cite{Ch-Pa-Su-2010} also showed that $H^*(M)\cong H^*(N)$ implies
$\beta_{i,j}(P)=\beta_{i,j}(Q)$ for all $i,j$, where $\beta_{i,j}(P)$ is the
\emph{$(i,j)$th graded Betti number} of $P$. One can define the graded Betti numbers
$\beta_{i,j}(P)$ using a finite free resolution of the face ring of $P$, for
example see \cite{Ch-Ki-2010}. Instead of doing this, we will simply take the
following Hochster's formula as the definition of $\beta_{i,j}(P)$:
\begin{equation}
  \label{eq:1}
\beta_{i,j}(P)=\sum_{\substack{W\subset V(P)\\ |W|=j}}
\dim_{\k}\widetilde H_{j-i-1}(P|_W;\k),
\end{equation}
where $V(P)$ is the set of vertices of $P$, $\k$ is an arbitrary field and
$P|_W$ is the realization of the simplicial complex $\{F \cap W : F \in
\Delta(P)\}$, where $\Delta(P)$ is the boundary complex of $P$.

With this motivation, we define the
following.

\begin{defn}
  A simplicial polytope $P$ is \emph{combinatorially rigid} (or simply
  \emph{rigid}) if we have $P=P'$ for any simplicial polytope $P'$ satisfying
  $\beta_{i,j}(P)=\beta_{i,j}(P')$ for all $i, j\geq0$.
\end{defn}

Hence, if $P$ supports a quasitoric manifold and $P$ is combinatorially rigid,
then $P$ is cohomologically rigid.

In the present paper, we investigate the combinatorial rigidity of
$3$-dimensional simplicial polytopes.  Remark that since all $3$-dimensional
simplicial polytopes support quasitoric manifolds, $3$-dimensional
combinatorially rigid polytopes are cohomologically rigid.

Let $P$ be a $3$-dimensional simplicial polytope with $n$ vertices. Then
$\dim_{\k}\widetilde H_{j-i-1}(P|_W;\k) = 0$ if $j-i-1\geq2$ for $W \subsetneq
V(P)$ and $\dim_{\k}\widetilde H_{j-i-1}(P|_W;\k)=\delta_{j-i-1,2}$ for
$W=V(P)$, where $\delta_{x,y}=1$ if $x=y$ and $0$ otherwise.  Thus it is enough
to consider $\beta_{i-1,i}(P)$ and $\beta_{i-2,i}(P)$. By the Poincar\'{e}
duality $\beta_{i,j}(P) = \beta_{n-3-i,n-j}(P)$, we have
$\beta_{i-2,i}(P)=\beta_{n-i-1,n-i}(P)$. Thus we only need to consider
$\beta_{i-1,i}(P)$, which we will call \emph{the $i$th special graded Betti
  number} and we denote $b_i(P):=\beta_{i-1,i}(P)$. By \eqref{eq:1}, we can
interpret $b_i(P)$ in a purely combinatorial way as follows:
\begin{equation}\label{eq:bk}
b_i(P)=\sum_{\substack{W\subset V(P)\\|W|=i}}
\left(\cc(P|_W)-1\right),
\end{equation}
where $\cc(P|_W)$ denotes the number of connected components of $P|_W$.

From the above observation, we get the following proposition.

\begin{prop}
  Let $P$ be a $3$-dimensional simplicial polytope. Then $P$ is combinatorially
  rigid if and only if $P$ is determined by $b_i(P)$'s for $i\geq 0$.
\end{prop}

So far, several polytopes are proved to be cohomologically rigid. In
$3$-dimensional case, Choi et al.~\cite{Ch-Pa-Su-2010} classified all
cohomologically rigid polytopes with at most $9$ vertices using computer and
proved that the icosahedron is cohomologically rigid. Since they only used the
graded Betti numbers, what they found are combinatorially rigid as well.

To state our main results we need to define connected sum which is a simple
operation to get a $d$-dimensional simplicial polytope from two $d$-dimensional
simplicial polytopes.

Let $P_1$ and $P_2$ be simplicial polytopes. A \emph{connected sum} of $P_1$ and
$P_2$ is a polytope obtained by attaching a facet of $P_1$ and a facet of
$P_2$. It depends on the way of choosing the two facets and identifying their
vertices. Let $\C(P_1\# P_2)$ denote the set of connected sums of $P_1$ and
$P_2$. If there is only one connected sum of $P_1$ and $P_2$ up to isomorphism,
then we will write the unique polytope as $P_1\# P_2$.

If a simplicial polytope $P$ can be expressed as a connected sum of two
polytopes, then $P$ is called \emph{reducible}. Otherwise, $P$ is called
\emph{irreducible}.

Let $T_4$, $C_8$, $O_6$, $D_{20}$ and $I_{12}$ be the five Platonic solids: the
tetrahedron, the cube, the octahedron, the dodecahedron and the icosahedron
respectively.

In this paper, we prove the following necessary condition to be combinatorially
rigid for $3$-dimensional reducible simplicial polytopes. See
Section~\ref{sec:necessary} or Figures~\ref{fig:cube}, \ref{fig:dodeca} and
\ref{fig:prism} for the definition of $\xi_1(C_8)$, $\xi_2(C_8) $,
$\xi_1(D_{20})$, $\xi_2(D_{20})$ and $B_n$, the bipyramid with $n$ vertices.

\begin{thm}\label{thm:necessary}
  Let $P$ be a $3$-dimensional simplicial polytope.  If $P$ is reducible and
  combinatorially rigid, then $P$ is either $T_4 \# T_4 \# T_4$ or $P_1\#P_2$,
  where
\begin{align*}
  P_1 & \in \{T_4,O_6,I_{12}\},\\
  P_2 &\in \{T_4,O_6,I_{12},\xi_1(C_8),\xi_2(C_8)
  ,\xi_1(D_{20}),\xi_2(D_{20})\} \cup \{B_n:n\geq7\}.
\end{align*}
\end{thm}

Note that $B_n$ is defined for $n\geq5$ and we have $B_5=T_4\#T_4$ and $B_6=O_6$.

In fact, $T_4 \# T_4 \# T_4$ is known to be rigid, see \cite{Ch-Pa-Su-2010}. We
also prove that $P_1\#P_2$ is rigid for some $P_1$ and $P_2$ in
Theorem~\ref{thm:necessary}.

\begin{thm}\label{thm:rigid}
  The following polytopes are combinatorially rigid:
$$T_4\# T_4, T_4\# O_6, T_4\# I_{12}, T_4\# B_n, O_6\# O_6, O_6\# B_n,$$
where $n\geq 7$. See Table~\ref{tab:rigid}.
\end{thm}

\begin{table}
  \begin{tabular}{c|c|c|c|c|c|c|c|c}
    $\#$ & $T_4$ & $O_6$ & $I_{12}$ & $B_n$, $n\geq7$ & $\xi_1(C_8)$& $\xi_2(C_8)$&
    $\xi_1(D_{20})$ & $\xi_2(D_{20})$\\ \hline
    $T_4$ & rigid & rigid & rigid & rigid & ?& ?& ?& ?\\ \hline
    $O_6$ & - & rigid & ? & rigid & ?& ?& ?& ?\\ \hline
    $I_{12}$ & - & - & ? & ? & ?& ?& ?& ?\\
 \end{tabular}
 \caption{Combinatorial rigidity of a connected sum of
   two irreducible polytopes.}
\label{tab:rigid}
\end{table}

The rest of this paper is organized as follows.  In Section~\ref{sec:necessary}
we prove Theorem~\ref{thm:necessary}. In Section~\ref{sec:max} we find the
maximum of $b_{n-4}(P)$ for a simplicial polytope $P$ with $n$ vertices.  In
Section~\ref{sec:rigid} we prove Theorem~\ref{thm:rigid}.

\section{A necessary condition for rigid reducible polytopes
}\label{sec:necessary}

The authors \cite{Ch-Ki-2010} found the following formula for the special graded
Betti numbers of $P\in\C(P_1\# P_2)$ for $d$-dimensional simplicial polytopes
$P_1$ and $P_2$ with $n_1$ and $n_2$ vertices respectively:
\begin{equation}\label{thm:connected_sum}
b_k (P) = \sum_{i=0}^k \left( b_i (P_1) \binom{n_2 -
  d}{k-i} + b_i (P_2) \binom{n_1 - d}{k-i}\right) + \binom{n_1 +
  n_2 -2d}{k}.
\end{equation}

The above formula says that the special graded Betti numbers of a connected sum
of two polytopes do not depend on the ways of choosing the two facets and
identifying them. Thus we get the following.

\begin{prop}\label{thm:unique}
  Let $P\in\C(P_1\# P_2)$ for $d$-dimensional simplicial polytopes $P_1$ and
  $P_2$. If $P$ is rigid, then $P$ is the only element in $\C(P_1\# P_2)$.
\end{prop}

\emph{From now on, all polytopes that we consider are $3$-dimensional and
  simplicial unless otherwise stated.}  As usual for $3$-dimensional polytopes,
we will call $0$, $1$, $2$-dimensional face, respectively, \emph{vertex},
\emph{edge} and \emph{face}.  We will sometimes identify a polytope $P$ with its
graph which is also called the $1$-skeleton of $P$. For a set $B$ of vertices,
$P|_B$ is the subgraph of $P$ induced by $B$.

Let $P$ be a polytope with vertex set $V$.  A \emph{$k$-belt} of $P$ is a set
$B=\{v_1,v_2,\ldots,v_k\}$ of $k$ vertices such that $P|_B$ is a $k$-gon and
$P|_{V\setminus B}$ is disconnected.  Let $|V|=n$. It is easy to see that if
$b_{n-k}(P)\ne 0$ for $k>0$, then $P$ has a $t$-belt for some $t\leq k$.

Note that $P$ has a $3$-belt if and only if $P$ is reducible. If
$P\in\C(P_1\#P_2)$, then the vertices of the attached face of $P_1$ (or
equivalently $P_2$) form a $3$-belt. Using this observation we can prove the
following.

\begin{prop}
  If $P$ is a connected sum of at least $3$ irreducible polytopes and $P\ne
  T_4\# T_4\# T_4$, then $P$ is not rigid.
\end{prop}
\begin{proof}
  Let $P=P_1 \# \cdots \# P_\ell$ for irreducible polytopes
  $P_1,\ldots,P_\ell$. By Proposition~\ref{thm:unique}, it is enough to show
  that there are two different polytopes in $\C(P_1 \# \cdots \#
  P_\ell)$.

  Let $Q\in\C(P_1 \# \cdots \# P_\ell)$ be a polytope satisfying the following
  condition $(*)$: there is an edge contained in all $3$-belts. We can construct
  such $Q$ as follows.  Let us fix an edge $\{a,b\}$ of $P_1$. Let $Q_1=P_1$ and
  for $2\leq i\leq \ell$, let $Q_i\in\C(Q_{i-1}\# P_i)$ be a polytope obtained
  by attaching a face of $Q_{i-1}$ containing $\{a,b\}$ and a face of
  $P_i$. Then $Q=Q_\ell$ satisfies the condition $(*)$.

  It is sufficient to construct $Q'\in\C(P_1 \# \cdots \# P_\ell)$ which
  does not satisfy the condition $(*)$.

  Assume $\ell=3$. Since $P\ne T_4\# T_4\# T_4$, we can assume that $P_1\ne
  T_4$.  Let $R\in\C(P_1\#P_2)$ and $\{a,b,c\}$ be the unique $3$-belt of
  $R$. Note that $R$ has at least $7$ faces. Since there are only $6$ faces in
  $R$ containing an edge of the $3$-belt $\{a,b,c\}$, we can find a face $F$ of
  $R$ which does not contain any such edge. Let $Q'$ be a polytope in
  $\C(R\#P_3)$ obtained by attaching $F$ and a face of $P_3$. Then the vertices
  of $F$ form a $3$-belt of $Q'$. Clearly $Q'$ does not satisfy the condition
  $(*)$.

 Assume $\ell\geq4$. Let $R\in\C(P_1 \# \cdots \# P_{\ell-1})$ be a
  polytope satisfying the condition $(*)$. Since $R$ has $\ell-2$ $3$-belts,
  there is a unique edge $\{a,b\}$ contained in all $3$-belts of $R$. Let
  $F$ be a face of $R$ which does not contain $\{a,b\}$. Let
  $Q'\in\C(R\#P_\ell)$ be a polytope obtained by attaching $F$ and a face of
  $P_\ell$. Then $Q'$ does not satisfy the condition $(*)$.
\end{proof}

Let $P$ be a polytope which is not necessarily simplicial. The
\emph{first-subdivision}, denoted $\xi_1(P)$, of $P$ is the simplicial polytope
obtained from $P$ by adding one vertex at the center of each face and connecting
it to all vertices of the face. The \emph{second-subdivision} (or
\emph{barycentric subdivision}) denoted $\xi_2(P)$, of $P$ is the simplicial
polytope obtained from $P$ by adding one vertex at the center of each face and
connecting it to all vertices of the face and all mid-points of the edges of the
face.  See Figures~\ref{fig:cube} and \ref{fig:dodeca}.

\begin{figure}
  \centering
  \begin{pspicture}(0,0)(3,1.5) \vput(0,1)0 \vput(0,0)0 \vput(1,0)0
    \vput(1.7,0.5)0 \vput(1.7,1.5)0 \vput(0.7,1.5)0
    \pspolygon(0,1)(0,0)(0,0)(1,0)(1,0)(1.7,0.5)(1.7,0.5)(1.7,1.5)(1.7,1.5)(0.7,1.5)
    \psline(1,1)(0,1) \psline(1,1)(1.7,1.5) \psline(1,1)(1,0)
\end{pspicture}
\begin{pspicture}(0,0)(3,1.5) \vput(0,1)0 \vput(0,0)0 \vput(1,0)0
  \vput(1.7,0.5)0 \vput(1.7,1.5)0 \vput(0.7,1.5)0
  \pspolygon(0,1)(0,0)(0,0)(1,0)(1,0)(1.7,0.5)(1.7,0.5)(1.7,1.5)(1.7,1.5)(0.7,1.5)
  \psline(1,1)(0,1) \psline(1,1)(1.7,1.5) \psline(1,1)(1,0)
  \vput(0.5,0.5)0 \psline(0.5,0.5)(1,1) \psline(0.5,0.5)(0,1)
  \psline(0.5,0.5)(0,0) \psline(0.5,0.5)(1,0) \vput(0.85,1.25)0
  \psline(0.85,1.25)(1,1) \psline(0.85,1.25)(0,1)
  \psline(0.85,1.25)(0.7,1.5) \psline(0.85,1.25)(1.7,1.5)
  \vput(1.35,0.75)0 \psline(1.35,0.75)(1,1) \psline(1.35,0.75)(1,0)
  \psline(1.35,0.75)(1.7,0.5) \psline(1.35,0.75)(1.7,1.5)
\end{pspicture}
\begin{pspicture}(0,0)(2,1.5) \vput(0,1)0 \vput(0,0)0 \vput(1,0)0
  \vput(1.7,0.5)0 \vput(1.7,1.5)0 \vput(0.7,1.5)0
  \pspolygon(0,1)(0,0)(0,0)(1,0)(1,0)(1.7,0.5)(1.7,0.5)(1.7,1.5)(1.7,1.5)(0.7,1.5)
  \psline(1,1)(0,1) \psline(1,1)(1.7,1.5) \psline(1,1)(1,0)
  \vput(0.5,0.5)0 \psline(0.5,0.5)(1,1) \psline(0.5,0.5)(0,1)
  \psline(0.5,0.5)(0,0) \psline(0.5,0.5)(1,0) \vput(0.5,1)0
  \psline(0.5,0.5)(0.5,1) \vput(0,0.5)0 \psline(0.5,0.5)(0,0.5)
  \vput(0.5,0)0 \psline(0.5,0.5)(0.5,0) \vput(1,0.5)0
  \psline(0.5,0.5)(1,0.5) \vput(0.85,1.25)0 \psline(0.85,1.25)(1,1)
  \psline(0.85,1.25)(0,1) \psline(0.85,1.25)(0.7,1.5)
  \psline(0.85,1.25)(1.7,1.5) \vput(0.5,1)0 \psline(0.85,1.25)(0.5,1)
  \vput(0.35,1.25)0 \psline(0.85,1.25)(0.35,1.25) \vput(1.2,1.5)0
  \psline(0.85,1.25)(1.2,1.5) \vput(1.35,1.25)0
  \psline(0.85,1.25)(1.35,1.25) \vput(1.35,0.75)0
  \psline(1.35,0.75)(1,1) \psline(1.35,0.75)(1,0)
  \psline(1.35,0.75)(1.7,0.5) \psline(1.35,0.75)(1.7,1.5)
  \vput(1,0.5)0 \psline(1.35,0.75)(1,0.5) \vput(1.35,0.25)0
  \psline(1.35,0.75)(1.35,0.25) \vput(1.7,1)0
  \psline(1.35,0.75)(1.7,1) \vput(1.35,1.25)0
  \psline(1.35,0.75)(1.35,1.25)
\end{pspicture}
 \caption{$C_8$, $\xi_1(C_8)$ and $\xi_2(C_8)$.}
  \label{fig:cube}
\end{figure}
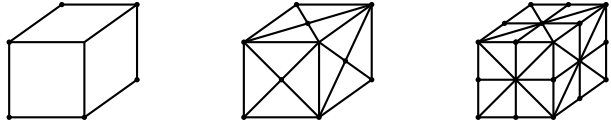

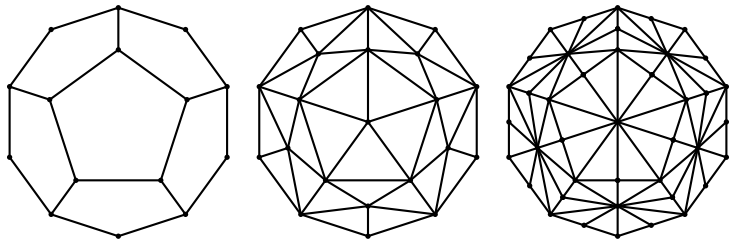
\begin{figure}
\centering
\psset{unit=.8cm}
\begin{pspicture}(-2,-2)(2,2) \vput(1.14127,0.37082)0 \vput(0,1.2)0
  \vput(-1.14127,0.37082)0 \vput(-0.705342,-0.97082)0
  \vput(0.705342,-0.97082)0
  \pspolygon(1.14127,0.37082)(0,1.2)(-1.14127,0.37082)(-0.705342,-0.97082)(0.705342,-0.97082)
  \vput(1.80701,0.587132)0 \vput(1.11679,1.53713)0 \vput(0,1.9)0
  \vput(-1.11679,1.53713)0 \vput(-1.80701,0.587132)0
  \vput(-1.80701,-0.587132)0 \vput(-1.11679,-1.53713)0 \vput(0,-1.9)0
  \vput(1.11679,-1.53713)0 \vput(1.80701,-0.587132)0
  \pspolygon(1.80701,0.587132)(1.11679,1.53713)(0,1.9)(-1.11679,1.53713)(-1.80701,0.587132)(-1.80701,-0.587132)(-1.11679,-1.53713)(0,-1.9)(1.11679,-1.53713)(1.80701,-0.587132)
  \psline(1.14127,0.37082)(1.80701,0.587132) \psline(0,1.2)(0,1.9)
  \psline(-1.14127,0.37082)(-1.80701,0.587132)
  \psline(-0.705342,-0.97082)(-1.11679,-1.53713)
  \psline(0.705342,-0.97082)(1.11679,-1.53713)
\end{pspicture}
\begin{pspicture}(-2,-2)(2,2) \vput(1.14127,0.37082)0 \vput(0,1.2)0
  \vput(-1.14127,0.37082)0 \vput(-0.705342,-0.97082)0
  \vput(0.705342,-0.97082)0
  \pspolygon(1.14127,0.37082)(0,1.2)(-1.14127,0.37082)(-0.705342,-0.97082)(0.705342,-0.97082)
  \vput(1.80701,0.587132)0 \vput(1.11679,1.53713)0 \vput(0,1.9)0
  \vput(-1.11679,1.53713)0 \vput(-1.80701,0.587132)0
  \vput(-1.80701,-0.587132)0 \vput(-1.11679,-1.53713)0 \vput(0,-1.9)0
  \vput(1.11679,-1.53713)0 \vput(1.80701,-0.587132)0
  \pspolygon(1.80701,0.587132)(1.11679,1.53713)(0,1.9)(-1.11679,1.53713)(-1.80701,0.587132)(-1.80701,-0.587132)(-1.11679,-1.53713)(0,-1.9)(1.11679,-1.53713)(1.80701,-0.587132)
  \psline(1.14127,0.37082)(1.80701,0.587132) \psline(0,1.2)(0,1.9)
  \psline(-1.14127,0.37082)(-1.80701,0.587132)
  \psline(-0.705342,-0.97082)(-1.11679,-1.53713)
  \psline(0.705342,-0.97082)(1.11679,-1.53713) \vput(0,0)0
  \psline(0,0)(1.14127,0.37082) \psline(0,0)(0,1.2)
  \psline(0,0)(-1.14127,0.37082) \psline(0,0)(-0.705342,-0.97082)
  \psline(0,0)(0.705342,-0.97082) \vput(0.822899,1.13262)0
  \psline(0.822899,1.13262)(0,1.2)
  \psline(0.822899,1.13262)(1.14127,0.37082)
  \psline(0.822899,1.13262)(1.80701,0.587132)
  \psline(0.822899,1.13262)(1.11679,1.53713)
  \psline(0.822899,1.13262)(0,1.9) \vput(-0.822899,1.13262)0
  \psline(-0.822899,1.13262)(-1.14127,0.37082)
  \psline(-0.822899,1.13262)(0,1.2) \psline(-0.822899,1.13262)(0,1.9)
  \psline(-0.822899,1.13262)(-1.11679,1.53713)
  \psline(-0.822899,1.13262)(-1.80701,0.587132)
  \vput(-1.33148,-0.432624)0
  \psline(-1.33148,-0.432624)(-0.705342,-0.97082)
  \psline(-1.33148,-0.432624)(-1.14127,0.37082)
  \psline(-1.33148,-0.432624)(-1.80701,0.587132)
  \psline(-1.33148,-0.432624)(-1.80701,-0.587132)
  \psline(-1.33148,-0.432624)(-1.11679,-1.53713) \vput(0,-1.4)0
  \psline(0,-1.4)(0.705342,-0.97082)
  \psline(0,-1.4)(-0.705342,-0.97082)
  \psline(0,-1.4)(-1.11679,-1.53713) \psline(0,-1.4)(0,-1.9)
  \psline(0,-1.4)(1.11679,-1.53713) \vput(1.33148,-0.432624)0
  \psline(1.33148,-0.432624)(1.14127,0.37082)
  \psline(1.33148,-0.432624)(0.705342,-0.97082)
  \psline(1.33148,-0.432624)(1.11679,-1.53713)
  \psline(1.33148,-0.432624)(1.80701,-0.587132)
  \psline(1.33148,-0.432624)(1.80701,0.587132)
\end{pspicture}
\begin{pspicture}(-2,-2)(2,2) \vput(1.14127,0.37082)0 \vput(0,1.2)0
  \vput(-1.14127,0.37082)0 \vput(-0.705342,-0.97082)0
  \vput(0.705342,-0.97082)0
  \pspolygon(1.14127,0.37082)(0,1.2)(-1.14127,0.37082)(-0.705342,-0.97082)(0.705342,-0.97082)
  \vput(1.80701,0.587132)0 \vput(1.11679,1.53713)0 \vput(0,1.9)0
  \vput(-1.11679,1.53713)0 \vput(-1.80701,0.587132)0
  \vput(-1.80701,-0.587132)0 \vput(-1.11679,-1.53713)0 \vput(0,-1.9)0
  \vput(1.11679,-1.53713)0 \vput(1.80701,-0.587132)0
  \pspolygon(1.80701,0.587132)(1.11679,1.53713)(0,1.9)(-1.11679,1.53713)(-1.80701,0.587132)(-1.80701,-0.587132)(-1.11679,-1.53713)(0,-1.9)(1.11679,-1.53713)(1.80701,-0.587132)
  \psline(1.14127,0.37082)(1.80701,0.587132) \psline(0,1.2)(0,1.9)
  \psline(-1.14127,0.37082)(-1.80701,0.587132)
  \psline(-0.705342,-0.97082)(-1.11679,-1.53713)
  \psline(0.705342,-0.97082)(1.11679,-1.53713) \vput(0,0)0
  \psline(0,0)(1.14127,0.37082) \psline(0,0)(0,1.2)
  \psline(0,0)(-1.14127,0.37082) \psline(0,0)(-0.705342,-0.97082)
  \psline(0,0)(0.705342,-0.97082) \vput(0.570634,0.78541)0
  \psline(0,0)(0.570634,0.78541) \vput(-0.570634,0.78541)0
  \psline(0,0)(-0.570634,0.78541) \vput(-0.923305,-0.3)0
  \psline(0,0)(-0.923305,-0.3) \vput(0,-0.97082)0
  \psline(0,0)(0,-0.97082) \vput(0.923305,-0.3)0
  \psline(0,0)(0.923305,-0.3) \vput(0.822899,1.13262)0
  \psline(0.822899,1.13262)(0,1.2)
  \psline(0.822899,1.13262)(1.14127,0.37082)
  \psline(0.822899,1.13262)(1.80701,0.587132)
  \psline(0.822899,1.13262)(1.11679,1.53713)
  \psline(0.822899,1.13262)(0,1.9) \vput(0.570634,0.78541)0
  \psline(0.822899,1.13262)(0.570634,0.78541) \vput(1.47414,0.478976)0
  \psline(0.822899,1.13262)(1.47414,0.478976) \vput(1.4619,1.06213)0
  \psline(0.822899,1.13262)(1.4619,1.06213) \vput(0.558396,1.71857)0
  \psline(0.822899,1.13262)(0.558396,1.71857) \vput(0,1.55)0
  \psline(0.822899,1.13262)(0,1.55) \vput(-0.822899,1.13262)0
  \psline(-0.822899,1.13262)(-1.14127,0.37082)
  \psline(-0.822899,1.13262)(0,1.2) \psline(-0.822899,1.13262)(0,1.9)
  \psline(-0.822899,1.13262)(-1.11679,1.53713)
  \psline(-0.822899,1.13262)(-1.80701,0.587132)
  \vput(-0.570634,0.78541)0
  \psline(-0.822899,1.13262)(-0.570634,0.78541) \vput(0,1.55)0
  \psline(-0.822899,1.13262)(0,1.55) \vput(-0.558396,1.71857)0
  \psline(-0.822899,1.13262)(-0.558396,1.71857)
  \vput(-1.4619,1.06213)0 \psline(-0.822899,1.13262)(-1.4619,1.06213)
  \vput(-1.47414,0.478976)0
  \psline(-0.822899,1.13262)(-1.47414,0.478976)
  \vput(-1.33148,-0.432624)0
  \psline(-1.33148,-0.432624)(-0.705342,-0.97082)
  \psline(-1.33148,-0.432624)(-1.14127,0.37082)
  \psline(-1.33148,-0.432624)(-1.80701,0.587132)
  \psline(-1.33148,-0.432624)(-1.80701,-0.587132)
  \psline(-1.33148,-0.432624)(-1.11679,-1.53713)
  \vput(-0.923305,-0.3)0 \psline(-1.33148,-0.432624)(-0.923305,-0.3)
  \vput(-1.47414,0.478976)0
  \psline(-1.33148,-0.432624)(-1.47414,0.478976) \vput(-1.80701,0)0
  \psline(-1.33148,-0.432624)(-1.80701,0) \vput(-1.4619,-1.06213)0
  \psline(-1.33148,-0.432624)(-1.4619,-1.06213)
  \vput(-0.911067,-1.25398)0
  \psline(-1.33148,-0.432624)(-0.911067,-1.25398) \vput(0,-1.4)0
  \psline(0,-1.4)(0.705342,-0.97082)
  \psline(0,-1.4)(-0.705342,-0.97082)
  \psline(0,-1.4)(-1.11679,-1.53713) \psline(0,-1.4)(0,-1.9)
  \psline(0,-1.4)(1.11679,-1.53713) \vput(0,-0.97082)0
  \psline(0,-1.4)(0,-0.97082) \vput(-0.911067,-1.25398)0
  \psline(0,-1.4)(-0.911067,-1.25398) \vput(-0.558396,-1.71857)0
  \psline(0,-1.4)(-0.558396,-1.71857) \vput(0.558396,-1.71857)0
  \psline(0,-1.4)(0.558396,-1.71857) \vput(0.911067,-1.25398)0
  \psline(0,-1.4)(0.911067,-1.25398) \vput(1.33148,-0.432624)0
  \psline(1.33148,-0.432624)(1.14127,0.37082)
  \psline(1.33148,-0.432624)(0.705342,-0.97082)
  \psline(1.33148,-0.432624)(1.11679,-1.53713)
  \psline(1.33148,-0.432624)(1.80701,-0.587132)
  \psline(1.33148,-0.432624)(1.80701,0.587132) \vput(0.923305,-0.3)0
  \psline(1.33148,-0.432624)(0.923305,-0.3) \vput(0.911067,-1.25398)0
  \psline(1.33148,-0.432624)(0.911067,-1.25398)
  \vput(1.4619,-1.06213)0 \psline(1.33148,-0.432624)(1.4619,-1.06213)
  \vput(1.80701,0)0 \psline(1.33148,-0.432624)(1.80701,0)
  \vput(1.47414,0.478976)0
  \psline(1.33148,-0.432624)(1.47414,0.478976)
\end{pspicture}
 \caption{$D_{20}$, $\xi_1(D_{20})$ and $\xi_2(D_{20})$.}
  \label{fig:dodeca}
\end{figure}

Let $P$ be a simplicial polytope. The \emph{type} of a face $F$ of $P$ is
defined to be $\type(F)=(x,y,z)$, where $x,y,z$ are the degrees of the three
vertices of $F$ with $x\geq y\geq z$.  If all faces of $P$ have the same type,
then $P$ is called \emph{face-transitive}\footnote{The usual definition is that
  $P$ is face-transitive if for any two faces $F_1$ and $F_2$ of $P$ there is an
  automorphism on $P$ sending $F_1$ to $F_2$. It is not difficult to see that
  our definition is equivalent to this.}  In this case, we define $\type(P)$ to
be the type of a face of $P$. If $P$ is face-transitive and $\type(P)=(x,x,x)$
for an integer $x$, then $P$ is called \emph{regular}. Note that $T_4$, $O_6$
and $I_{12}$ are the only regular simplicial polytopes.

\begin{lem}\label{lem:face-transitive}
  If $|\C(P\#Q)|=1$ for irreducible polytopes $P$ and $Q$, then one of $P$, $Q$
  is regular and the other is face-transitive.
\end{lem}
\begin{proof}
  If $P$ is not face-transitive, then $P$ has two faces $F_1$, $F_2$ with
  different types.  Let $F$ be any face of $Q$. Let $P_1$ (resp.~$P_2$) be a
  polytope in $\C(P\#Q)$ obtained by identifying $F_1$ (resp.~$F_2$) with
  $F$. Then $P_1$ and $P_2$ can not be the same, which is a contradiction to
  $|\C(P\#Q)|=1$. Thus $P$ is face-transitive and so is $Q$ by the same
  argument.

  Let $\type(P)=(x,y,z)$ and $\type(Q)=(a,b,c)$. We can identify a face $F'$ of
  $P$ with a face $F''$ of $Q$ in the following two ways: identify the vertices
  of degree $x,y,z$ in $F'$ with (1) the vertices of degree $a,b,c$ in $F''$ and
  (2) the vertices of degree $c,b,a$ in $F''$ respectively. Then the resulting
  polytope has a unique $3$-belt with vertices of degree $x+a-2, y+b-2, z+c-2$
  in the first case and $x+c-2,y+b-2,z+a-2$ in the second case. Since two
  polytopes are the same, we have $x+a-2=x+c-2$ or $x+a-2=z+a-2$. Thus $a=c$ or
  $x=z$. Since $a\geq b\geq c$ and $x\geq y\geq z$, we have $a=b=c$ or $x=y=z$,
  which implies that either $P$ or $Q$ is regular.
\end{proof}

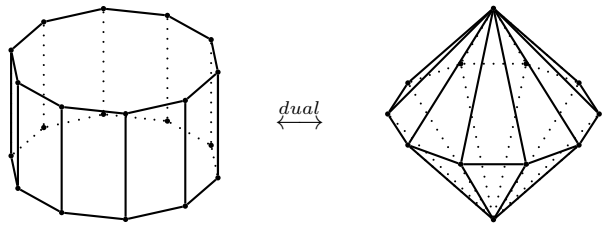
\begin{figure}
  \centering
\psset{unit=20pt}
\begin{pspicture}(-2,-2)(5,2)
\rput(3.5,0){$\stackrel{dual}{\longleftrightarrow}$}
\vput(1.9563,0.792088){x0} \vput(1.9563,-1.20791){y0} \vput(1.33826,0.256855){x1} \vput(1.33826,-1.74314){y1} \vput(0.209057,0.0054781){x2} \vput(0.209057,-1.99452){y2} \vput(-1,0.133975){x3} \vput(-1,-1.86603){y3} \vput(-1.82709,0.593263){x4} \vput(-1.82709,-1.40674){y4} \vput(-1.9563,1.20791){x5} \vput(-1.9563,-0.792088){y5} \vput(-1.33826,1.74314){x6} \vput(-1.33826,-0.256855){y6} \vput(-0.209057,1.99452){x7} \vput(-0.209057,-0.0054781){y7} \vput(1,1.86603){x8} \vput(1,-0.133975){y8} \vput(1.82709,1.40674){x9} \vput(1.82709,-0.593263){y9} \edge{x0}{x1} \edge{x1}{x2} \edge{x2}{x3} \edge{x3}{x4} \edge{x4}{x5} \edge{x5}{x6} \edge{x6}{x7} \edge{x7}{x8} \edge{x8}{x9} \edge{x9}{x0} \edge{y0}{y1} \edge{x0}{y0} \edge{y1}{y2} \edge{x1}{y1} \edge{y2}{y3} \edge{x2}{y2} \edge{y3}{y4} \edge{x3}{y3} \edge{y4}{y5} \edge{x4}{y4} \edge{x5}{y5} \edge{x0}{y0} \psset{linestyle=dotted}\edge{y5}{y6} \edge{x5}{y5} \edge{y6}{y7} \edge{x6}{y6} \edge{y7}{y8} \edge{x7}{y7} \edge{y8}{y9} \edge{x8}{y8} \edge{y9}{y0} \edge{x9}{y9}
\end{pspicture}
\begin{pspicture}(-2,-2)(2,2)
\vput(0,2){x} \vput(0,-2){y} \vput(2,0){0} \vput(1.61803,-0.587785){1} \vput(0.618034,-0.951057){2} \vput(-0.618034,-0.951057){3} \vput(-1.61803,-0.587785){4} \vput(-2,0){5} \vput(-1.61803,0.587785){6} \vput(-0.618034,0.951057){7} \vput(0.618034,0.951057){8} \vput(1.61803,0.587785){9} \edge{8}{-1} \edge{9}{0} \edge{0}{1} \edge{1}{2} \edge{2}{3} \edge{3}{4} \edge{4}{5} \edge{5}{6} \edge{x}{1} \edge{y}{1} \edge{x}{2} \edge{y}{2} \edge{x}{3} \edge{y}{3} \edge{x}{4} \edge{y}{4} \edge{9}{x} \edge{5}{x} \edge{0}{x} \edge{6}{x} \psset{linestyle=dotted} \edge{4}{y} \edge{5}{y} \edge{6}{y} \edge{7}{y} \edge{8}{y} \edge{9}{y} \edge{0}{y} \edge{1}{y} \edge{7}{x} \edge{6}{7} \edge{8}{x} \edge{7}{8} \edge{9}{x} \edge{8}{9}
\end{pspicture}
 \caption{A prism and a bipyramid. They are dual to each other.}
  \label{fig:prism}
\end{figure}

A \emph{prism} is the product of an $n$-gon and an interval. A \emph{bipyramid}
is the dual of a prism. Let $B_n$ denote the bipyramid with $n$ vertices. See
Figure~\ref{fig:prism}.

Fleischner and Imrich \cite[Theorem 3]{Fleischner1979} classified all
face-transitive $3$-dimensional polytopes. The following is a consequence of
their result.

\begin{prop}
  Let $P$ be a face-transitive simplicial polytope. Then $P$ is either
  a bipyramid, a Platonic solid, the first-subdivision of a Platonic
  solid or the second-subdivision of a Platonic solid.
\end{prop}

One can check that $\xi_2(T_4) = \xi_1(C_8)$, $\xi_2(C_8) = \xi_2(O_6)$ and
$\xi_2(D_{20})=\xi_2(I_{12})$. Note that $\xi_1(T_4)$, $\xi_1(O_6)$ and
$\xi_1(I_{12})$ are reducible. Thus we get Theorem~\ref{thm:necessary}.

\section{The maximum of $b_{n-4}(P)$ for irreducible polytopes} \label{sec:max}

Let $P$ be a polytope with $n$ vertices. Since $b_{n-3}(P)$ is the number of
$3$-belts, we get the following.

\begin{prop}
  Let $P_1,\ldots,P_\ell$ be irreducible polytopes and let
  $P\in\C(P_1\#\cdots\#P_\ell)$.  If $P$ has $n$ vertices, then
  $b_{n-3}(P)=\ell-1$.
\end{prop}

We will first find an upper bound of $b_{n-4}(P)$ for an irreducible polytope
$P$ with $n$ vertices, which is established when $P=B_n$, the bipyramid with
$n$ vertices. Our first step is to find $b_k(B_n)$.

\begin{prop}
For $k\leq n-3$, we have
$$b_k(B_n) =\frac{(n-2)(k-1)}{n-2-k}\binom{n-4}{k}+\delta_{k,2}.$$
\end{prop}
\begin{proof}
  Observe that $B_n$ is the graph obtained from an $(n-2)$-gon by adding two
  vertices $v$, $u$ connected to all vertices of the $(n-2)$-gon. Let $W$ be a
  set of $k$ vertices.  If $v\in W$ or $u\in W$, then $B_n|_W$ is connected
  unless $k=2$ and $W=\{u,v\}$.  Thus it is sufficient to consider the vertices
  in the $(n-2)$-gon. Then it follows from the result in
  \cite[Example~2.1.(c)]{Br-Hi-1998}; see also \cite[Corollary~3.7]{Ch-Ki-2010}.
\end{proof}

\begin{table}
\psset{unit=4pt}
\centering
\begin{tabular}{|c|c|c|c|c|c|c|c|c|c|c|} \hline
$n$ & $4$ & $6$ & $7$ &
\multicolumn{2}{c|}{$8$} &
\multicolumn{5}{c|}{$9$} \\
\hline
$P^*$
& \raisebox{-9pt}{\begin{pspicture}(0,-1)(6,6.5)
\rput{180}(3,2.4){\rput(-3,-3){\pspolygon(3,0)(0.4,4.5)(5.6,4.5) \psline(0.4,4.5)(3,3)
\psline(3,0)(3,3) \psline(5.6,4.5)(3,3)}}
\end{pspicture}}
& \raisebox{-9pt}{\begin{pspicture}(0,-1)(6,6.5)
\pspolygon(5.12,5.12)(0.88,5.12)(0.88,0.88)(5.12,0.88)
\pspolygon(4.06,4.06)(1.94,4.06)(1.94,1.94)(4.06,1.94)
\psline(5.12,5.12)(4.06,4.06) \psline(0.88,5.12)(1.94,4.06)
\psline(0.88,0.88)(1.94,1.94) \psline(5.12,0.88)(4.06,1.94)
\end{pspicture}}
& \raisebox{-9pt}{\begin{pspicture}(0,-1)(6,6.5)
\pspolygon(3,6)(0.15,3.93)(1.24,0.57)(4.76,0.57)(5.85,3.93)
\pspolygon(3,4.5)(1.57,3.46)(2.12,1.79)(3.88,1.79)(4.43,3.46)
\psline(3,6)(3,4.5) \psline(0.15,3.93)(1.57,3.46)
\psline(1.24,0.57)(2.12,1.79) \psline(4.76,0.57)(3.88,1.79)
\psline(5.85,3.93)(4.43,3.46)
\end{pspicture}}
& \raisebox{-9pt}{\begin{pspicture}(0,-1)(6,6.5)
\pspolygon(0,3)(1.5,0.4)(4.5,0.4)(6,3)(4.5,5.6)(1.5,5.6)
\pspolygon(1.5,3)(2.25,1.7)(3.75,1.7)(4.5,3)(3.75,4.3)(2.25,4.3)
\psline(0,3)(1.5,3) \psline(1.5,0.4)(2.25,1.7)
\psline(4.5,0.4)(3.75,1.7) \psline(6,3)(4.5,3)
\psline(4.5,5.6)(3.75,4.3) \psline(1.5,5.6)(2.25,4.3)
\end{pspicture}}
& \raisebox{-9pt}{\begin{pspicture}(0,-1)(6,6.5)
\pspolygon(3,6)(0.15,3.93)(1.24,0.57)(4.76,0.57)(5.85,3.93)
\pspolygon(3,4.5)(1.57,3.46)(2.12,2.09)(3.88,2.09)(4.43,3.46)
\pspolygon(2.12,2.09)(3.88,2.09)(3.88,1.49)(2.12,1.49)
\psline(3,6)(3,4.5) \psline(0.15,3.93)(1.57,3.46)
\psline(1.24,0.57)(2.12,1.49) \psline(4.76,0.57)(3.88,1.49)
\psline(5.85,3.93)(4.43,3.46)
\end{pspicture}}
& \raisebox{-9pt}{\begin{pspicture}(0,-1)(6,6.5)
\pspolygon(3,6)(0.65,4.87)(0.08,2.33)(1.69,0.30)(4.31,0.30)(5.92,2.33)(5.35,4.87)
\pspolygon(3,4.5)(1.83,3.93)(1.54,2.67)(2.35,1.65)(3.65,1.65)(4.46,2.67)(4.17,3.93)
\psline(3,6)(3,4.5) \psline(0.65,4.87)(1.83,3.93)
\psline(0.08,2.33)(1.54,2.67) \psline(1.69,0.30)(2.35,1.65)
\psline(4.31,0.30)(3.65,1.65) \psline(5.92,2.33)(4.46,2.67)
\psline(5.35,4.87)(4.17,3.93)
\end{pspicture}}
& \raisebox{-9pt}{\begin{pspicture}(0,-1)(6,6.5)
\pspolygon(0,3)(1.5,0.4)(4.5,0.4)(6,3)(4.5,5.6)(1.5,5.6)
\pspolygon(1,3)(1.5,3.87)(2,3)(1.5,2.13)
\pspolygon(4,3)(4.5,3.87)(5,3)(4.5,2.13) \psline(2,3)(4,3)
\psline(0,3)(1,3) \psline(5,3)(6,3) \psline(1.5,0.4)(1.5,2.13)
\psline(1.5,5.6)(1.5,3.87) \psline(4.5,0.4)(4.5,2.13)
\psline(4.5,5.6)(4.5,3.87)
\end{pspicture}}
& \raisebox{-9pt}{\begin{pspicture}(0,-1)(6,6.5)
\pspolygon(0,3)(1.5,0.4)(4.5,0.4)(6,3)(4.5,5.6)(1.5,5.6)
\pspolygon(1.5,3)(2.25,2)(3.75,2)(4.5,3)(3.75,4.3)(2.25,4.3)
\pspolygon(2.25,1.4)(3.75,1.4)(3.75,2)(2.25,2) \psline(0,3)(1.5,3)
\psline(1.5,0.4)(2.25,1.4) \psline(4.5,0.4)(3.75,1.4)
\psline(6,3)(4.5,3) \psline(4.5,5.6)(3.75,4.3)
\psline(1.5,5.6)(2.25,4.3)
\end{pspicture}}
& \raisebox{-9pt}{\begin{pspicture}(0,-1)(6,6.5)
\pspolygon(0,3)(1.5,0.4)(4.5,0.4)(6,3)(4.5,5.6)(1.5,5.6)
\pspolygon(1.5,3)(2.25,1.7)(3,2)(3.75,1.7)(4.5,3)(3.75,4.3)(3,4)(2.25,4.3)
\psline(3,2)(3,4) \psline(0,3)(1.5,3) \psline(1.5,0.4)(2.25,1.7)
\psline(4.5,0.4)(3.75,1.7) \psline(6,3)(4.5,3)
\psline(4.5,5.6)(3.75,4.3) \psline(1.5,5.6)(2.25,4.3)
\end{pspicture}}
& \raisebox{-9pt}{\begin{pspicture}(0,-1)(6,6.5)
\pspolygon(3,6)(0.15,3.93)(1.24,0.57)(4.76,0.57)(5.85,3.93)
\pspolygon(3,4.5)(2.49,3.70)(1.57,3.46)(2.12,1.79)(3,2.24)(3.88,1.79)(4.43,3.46)(3.52,3.66)
\psline(3,6)(3,4.5) \psline(0.15,3.93)(1.57,3.46)
\psline(1.24,0.57)(2.12,1.79) \psline(4.76,0.57)(3.88,1.79)
\psline(5.85,3.93)(4.43,3.46) \psline(3,3)(3,2.24)
\psline(3,3)(2.49,3.70) \psline(3,3)(3.52,3.66)
\end{pspicture}}
\\  \hline
$b_{n-4}(P)$ & $-1$ & $3$ & $5$ & $9$ & $5$ & $14$ & $12$ & $8$ & $6$ & $3$
\\  \hline
\end{tabular}
\caption{The complete list of simple polytopes $P^*$ with $n$
faces
for $n\leq 9$ such that $P$ is an irreducible simplicial polytope
with $n$ vertices,
and the  numbers $b_{n-4}(P)$.} \label{table:list}
\end{table}

Choi et al. \cite{Ch-Pa-Su-2010} computed the graded Betti numbers of all
simplicial polytopes with at most $9$ vertices. We need some of their result as
shown in Table~\ref{table:list}.

\begin{thm}\label{thm:irr n-4}
  Let $n\geq 4$ be a fixed integer.  Let $P$ be an irreducible polytope with $n$
  vertices.  Then
$$b_{n-4}(P)\leq \binom{n-3}{2}-1 + \delta_{n,6}.$$ The equality
holds if and only if $P=B_n$ or $P=T_4$.
\end{thm}
\begin{proof}
  Induction on $n$. By Table~\ref{table:list}, it is true for $n\leq 7$.  Assume
  that $n\geq 8$ and it is true for all integers less than $n$.  Since $P$ is
  irreducible, $b_{n-4}(P)$ is the number of $4$-belts of $P$. If there is no
  4-belt in $P$, the theorem is true since $b_{n-4}(P)=0$ and $P$ is not a
  bipyramid. Otherwise, take a 4-belt $B=\{v_1,v_2,v_3,v_4\}$ such that $v_i$ is
  connected to $v_{i+1}$ for $i=1,2,3,4$, where $v_5=v_1$.

Now assume that the graph $P$ is embedded in a plane. Let $V$ be the vertex set
of $P$. There are two connected components in $P|_{V\setminus B}$. Let $X_1$
(resp.~$X_2$) be the set of vertices in the connected component in
$P|_{V\setminus B}$ which is outside (resp.~inside) of the $4$-gon consisting of
the vertices in $B$.

We can assume that $|X_1|>1$ and $|X_2|>1$, because otherwise we can assign the
unique vertex of $X_1$ or $X_2$ to $B$, which implies that $b_{n-4}$ is less
than the number of vertices, and hence, $b_{n-4}(P)\leq n < \binom{n-3}2-1$.
For $i=1,2$, let $P_i$ be the polytope obtained from $P$ by contracting all
vertices in $X_i$ to a single vertex $x_i$. Note that $x_i$ is connected to all
vertices of $B$ in $P_i$.

Let $n_1$ and $n_2$ be the number of vertices of $P_1$ and $P_2$
respectively. Then $n_1+n_2=n+6$ and $n_1,n_2\geq 7$.

For $i=1,2$, let $A_i$ (resp.~$B_i$) be the set of vertices $u$ of $P_i$ such
that $\{x_i,v_1,v_3,u\}$ (resp.~$\{x_i,v_2,v_4,u\}$) is a $4$-belt. We claim
that $A_1 =\emptyset$ or $B_1 =\emptyset$. Assume that both $A_1$ and $B_1$ are
nonempty. Note that $x_1$ is the only vertex in $P_1$ which lies outside of the
$4$-gon $\{v_1,v_2,v_3,v_4\}$. If $u\in A_1$ and $u'\in B_1$, then $\{u,v_1\}$,
$\{u,v_3\}$, $\{u',v_2\}$ and $\{u',v_4\}$ are edges. Since $P_1$ is a planar
graph, we must have $u=u'$. Then the edges $\{u,v_i\}$ for $i=1,2,3,4$ divide
the $4$-gon $\{v_1,v_2,v_3,v_4\}$ into four triangular regions.  If we have a vertex inside of the
triangle $\{u,v_i,v_{i+1}\}$, then this forms a $3$-belt of $P_1$, and thus of
$P$. Thus we do not have any vertex except $u$ inside of the $4$-gon
$\{v_1,v_2,v_3,v_4\}$. Then we get $n_1=6$ which is a contradiction. Thus we
have $A_1 =\emptyset$ or $B_1 =\emptyset$. For the same reason, we also have
$A_2 =\emptyset$ or $B_2 =\emptyset$.

Let $a_i=|A_i|$ and $b_i=|B_i|$ for $i=1,2$. Then we have
\begin{align*}
b_{n-4}(P)&=b_{n_1-4}(P_1)+b_{n_1-4}(P_1)-1-a_1-a_2-b_1-b_2+a_1a_2+b_1b_2\\
&=b_{n_1-4}(P_1)+b_{n_1-4}(P_1)+(a_1-1)(a_2-1)+(b_1-1)(b_2-1)-3.
\end{align*}
Since $a_1+b_1\leq n_1-5$ and $a_2+b_2\leq n_2-5$, we have
$(a_1-1)(a_2-1)+(b_1-1)(b_2-1)\leq (n_1-6)(n_2-6)+1$, where the equality holds
if and only if $(a_1,b_1,a_2,b_2)$ is equal to $(n_1-5,0,n_2-5,0)$ or
$(0,n_1-5,0,n_2-5)$.  Note that $a_i=n_i-5$ or $b_i=n_i-5$ if and only if $P_i$
is a bipyramid. Moreover, $(a_1,b_1,a_2,b_2)$ is equal to $(n_1-5,0,n_2-5,0)$ or
$(0,n_1-5,0,n_2-5)$ if and only if $P=B_n$.

Since $n_1,n_2<n$, by the
induction hypothesis, we get
\begin{align*}
b_{n-4}(P)&\leq \binom{n_1-3}2-1+\binom{n_2-3}2-1 +(n_1-6)(n_2-6)+1-3\\
&=\binom{n-3}2-1,
\end{align*}
where the equality holds if and only if $P=B_n$.
\end{proof}

Using a similar argument, we can find the second largest value of
$b_{n-4}(P)$ for an irreducible polytope $P$ with $n$ vertices.

\begin{figure}
  \centering
\psset{unit=20pt}
\begin{pspicture}(-2,-2)(5,2)
\rput(3.5,0){$\stackrel{dual}{\longleftrightarrow}$}
\vput(1.9563,0.792088){x0} \vput(1.9563,-1.20791){y0}
\vput(1.33826,0.256855){x1} \vput(1.33826,-1.74314){y1}
\vput(0.434898,0.0557535){x2} \vput(-1.16542,0.225833){x3}
\vput(0.209057,-.3445219){xx2} \vput(-1,-.216025){xx3} \edge{xx2}{x2}
\edge{xx3}{x3} \edge{xx3}{xx2} \vput(0.209057,-1.99452){y2}
\vput(-1,-1.86603){y3} \vput(-1.82709,0.593263){x4}
\vput(-1.82709,-1.40674){y4} \vput(-1.9563,1.20791){x5}
\vput(-1.9563,-0.792088){y5} \vput(-1.33826,1.74314){x6}
\vput(-1.33826,-0.256855){y6} \vput(-0.209057,1.99452){x7}
\vput(-0.209057,-0.0054781){y7} \vput(1,1.86603){x8}
\vput(1,-0.133975){y8} \vput(1.82709,1.40674){x9}
\vput(1.82709,-0.593263){y9} \edge{x0}{x1} \edge{x1}{x2} \edge{x2}{x3}
\edge{x3}{x4} \edge{x4}{x5} \edge{x5}{x6} \edge{x6}{x7} \edge{x7}{x8}
\edge{x8}{x9} \edge{x9}{x0} \edge{y0}{y1} \edge{x0}{y0} \edge{y1}{y2}
\edge{x1}{y1} \edge{y2}{y3} \edge{xx2}{y2} \edge{y3}{y4}
\edge{xx3}{y3} \edge{y4}{y5} \edge{x4}{y4} \edge{x5}{y5} \edge{x0}{y0}
\psset{linestyle=dotted}\edge{y5}{y6} \edge{x5}{y5} \edge{y6}{y7}
\edge{x6}{y6} \edge{y7}{y8} \edge{x7}{y7} \edge{y8}{y9} \edge{x8}{y8}
\edge{y9}{y0} \edge{x9}{y9}
\end{pspicture}
\begin{pspicture}(-2,-2)(2,2)
\vput(0,2){x} \vput(0,-2){y} \vput(2,0){0} \vput(1.61803,-0.587785){1}
\vput(0.618034,-0.951057){2} \vput(-0.618034,-0.951057){3}
\vput(-1.61803,-0.587785){4} \vput(-2,0){5}
\vput(-1.61803,0.587785){6} \vput(-0.618034,0.951057){7}
\vput(0.618034,0.951057){8} \vput(1.61803,0.587785){9}
\vput(-.3,.5){z}
\edge zx \edge z4 \edge z3 \edge z2
\edge{8}{-1} \edge{9}{0} \edge{0}{1} \edge{1}{2} \edge{2}{3}
\edge{3}{4} \edge{4}{5} \edge{5}{6} \edge{x}{1} \edge{y}{1}
\edge{x}{2} \edge{y}{2} \edge{y}{3} \edge{x}{4} \edge{y}{4}
\edge{9}{x} \edge{5}{x} \edge{0}{x} \edge{6}{x}
\psset{linestyle=dotted}
\edge{4}{y} \edge{5}{y} \edge{6}{y} \edge{7}{y} \edge{8}{y}
\edge{9}{y} \edge{0}{y} \edge{1}{y} \edge{7}{x} \edge{6}{7}
\edge{8}{x} \edge{7}{8} \edge{9}{x} \edge{8}{9}
\end{pspicture}
\caption{An edge-cut-prism and a semi-bipyramid. They are dual to each other.}
  \label{fig:edgecut}
\end{figure}
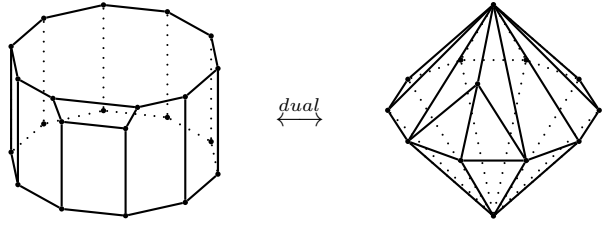

Let $P$ be a prism which is a product of a $k$-gon and an interval. Let $e$ be
an edge of one of the two $k$-gons of $P$. Then we can obtain another simple
polytope from $P$ by `cutting' the edge $e$. We will call such a polytope an
\emph{edge-cut-prism}. A \emph{semi-bipyramid} is the dual of an edge-cut-prism.
See Figure~\ref{fig:edgecut}.

\begin{thm}\label{thm:second max}
  Let $P$ be an irreducible polytope with $n$ vertices.  If $P\ne B_n$, then
$$b_{n-4}(P)\leq \binom{n-5}{2}+2.$$ The equality holds if and only if $P$
is a semi-bipyramid.
\end{thm}
\begin{proof}
  Induction on $n$. If $n\leq 9$, then by Table~\ref{table:list}, it holds.

  Assume that $n\geq 10$ and it holds for all integers less than $n$.  We will
  use the same notations in the proof of Theorem~\ref{thm:irr n-4}.

  \begin{description}
  \item[Case 1] $a_i= n_i-5$ or $b_i= n_i-5$ for $i=1,2$. Then $P_1$ and $P_2$
    are bipyramids. Recall that if $(a_1,b_1,a_2,b_2)$ is equal to
    $(n_1-5,0,n_2-5,0)$ or $(0,n_1-5,0,n_2-5)$, then $P=B_n$. Thus
    $(a_1,b_1,a_2,b_2)$ must be equal to $(n_1-5,0,0,n_2-5)$ or
    $(0,n_1-5,n_2-5,0)$. Hence,
\begin{align*}
b_{n-4}(P) &= b_{n_1-4}(P_1)+b_{n_1-4}(P_1)-1-a_1-a_2-b_1-b_2+a_1a_2+b_1b_2\\
&= \binom{n_1-3}2 -1+\binom{n_2-3}2 -1  -1 -(n_1-5)-(n_2-5)\\
&=\binom{n_1-3}2 +\binom{n_2-3}2 -n+1.
\end{align*}
Recall that $n\geq 10$, $n_1,n_2\geq7$ and $n_1+n_2=n+6$. It is easy to check
that if $n_1=7$ or $n_2=7$, then $P$ is a semi-bipyramid and
$b_{n-4}(P)=\binom{n-5}{2}+2$.  Otherwise $\binom{n_1-3}2 +\binom{n_2-3}2
-n+1\leq \binom{n-5}{2}+\binom{5}{2} -n+1< \binom{n-5}{2}+2$ because $n\geq10$.

\item[Case 2] Otherwise. We can assume that $1\leq a_1< n_1-5$. We claim that
  $a_1\ne n_1-6$. For contradiction, suppose $a_1 = n_1-6$. Let
  $A_1=\{u_1,u_2,\ldots,u_{n_1-6}\}$ and assume that in the embedding of $P_1$,
  the $4$-gon $\{v_1,v_2,v_3,v_4\}$ is divided into the $4$-gons
  $\{v_1,u_i,v_3,u_{i+1}\}$ for $i=0,1,\ldots,n_1-6$, where $u_0=v_2$ and
  $u_{n_1-5}=v_4$. Let $w$ be the unique vertex inside of the $4$-gon
  $\{v_1,v_2,v_3,v_4\}$ which is not contained in $A_1$. Then $w$ is contained
  in the $4$-gon $\{v_1,u_i,v_3,u_{i+1}\}$ for some $i$. If $\{u_i,u_{i+1}\}$ is
  an edge, then $\{v_1,u_i,u_{i+1}\}$ or $\{v_3,u_i,u_{i+1}\}$ is a
  $3$-belt. Thus $\{u_i,w\}$ and $\{u_{i+1},w\}$ are edges. Since $P_1$ is
  simplicial, $\{v_1,w\}$ and $\{v_3,w\}$ are also edges. Then
  $\{v_1,w,v_3,x_1\}$ is a $4$-belt of $P_1$, which is a contradiction to
  $w\not\in A_1$. Thus we must have $1\leq a_1\leq n_1-7$. Using a similar
  argument, we can check that $a_1=n_1-7$ if and only if $P_1$ is a
  semi-bipyramid.

  Since $P_1$ is not a bipyramid and $n_1<n$, by the induction
  hypothesis, we get
\begin{align*}
b_{n-4}(P) &=b_{n_1-4}(P_1)+b_{n_1-4}(P_1)-1-a_1-a_2-b_1-b_2+a_1a_2+b_1b_2\\
&\leq b_{n_1-4}(P_1)+b_{n_2-4}(P_2)-1- a_1-a_2+a_1a_2\\
&= b_{n_1-4}(P_1)+b_{n_2-4}(P_2)+(a_1-1)(a_2-1)-2\\
 &\leq \binom{n_1-5}{2}+2+\binom{n_2-3}{2}-1+(n_1-8)(n_2-6)-2\\
&=\binom{n-5}{2}+2,
\end{align*}
 \end{description}
where the equality holds if and only if $a_1=n_1-7$ and $a_2=n_1-5$,
equivalently, $P$ is a semi-bipyramid.
\end{proof}

Now we will find the maximum of $b_{n-4}(P)$ when $P$ is a connected sum of
$\ell$ irreducible polytopes.

Let $P\in\C(P_1\# P_2)$ be a polytope with $n$ vertices, where $P_1$ and $P_2$
are polytopes with $n_1$ and $n_2$ vertices respectively. Additionally, we
assume that $P_2$ is irreducible.  Note that $n=n_1+n_2-3$. By
\eqref{thm:connected_sum}, we have
$$b_{n-4} (P) = \sum_{i=0}^{n-4} \left( b_i (P_1) \binom{n_2 -
    3}{n-4-i} + b_i (P_2) \binom{n_1 - 3}{n-4-i}\right) + \binom{n-3}{n-4}.$$
Since $\binom{n_2 - 3}{n-4-i}=0$ unless $i\geq n_1-4$ and $\binom{n_1 -
  3}{n-4-i}=0$ unless $i\geq n_2-4$, we get
\begin{equation}\label{eq:n-4}
  b_{n-4}(P)=b_{n_1-4}(P_1)+b_{n_2-4}(P_2)+(n_2-3)b_{n_1-3}(P_1) +(n-3).
\end{equation}

\begin{lem}\label{thm:mainlemma}
  Let $P\in\C(P_1\#\cdots \# P_\ell)$, where $P_i$ is an irreducible polytope
  with $n_i$ vertices. Let $n$ be the number of vertices of $P$. Then
$$b_{n-4}(P) = \sum_{i=1}^{\ell} b_{n_i  -4}(P_i) + (n-3)(\ell-1).$$
\end{lem}
\begin{proof}
  Induction on $\ell$. If $\ell=1$ then it is clear.  Let $\ell\geq2$ and assume
  that it holds for all integers less than $\ell$.

Let $P\in\C(P_1\#\cdots\#P_\ell)$. Then $P\in\C(Q\# P_\ell)$ for
some $Q\in\C(P_1\#\cdots\#P_{\ell-1})$. Let $n'$ be the number of
vertices of $Q$. Then by \eqref{eq:n-4},
$$b_{n-4}(P)=b_{n'-4}(Q)+b_{n_\ell-4}(P_\ell)+(n_\ell-3)b_{n'-3}(Q)+(n-3).$$
By the induction hypothesis,
$$b_{n'-4}(Q)=\sum_{i=1}^{\ell-1}b_{n_i-4}(P_i)+(n'-3)(\ell-2).$$ Since
$b_{n'-3}(Q)=\ell-2$ and $(n'-3)+(n_\ell-3)=n-3$, we get
$$b_{n-4}(P)=\sum_{i=1}^{\ell}b_{n_i-4}(P_i)+(n-3)(\ell-1).$$
\end{proof}

For an integer $n$, let $f(n)=\binom{n-3}{2}-1+\delta_{n,6}$. Thus $f(n)$ is the
maximum of $b_{n-4}(P)$ as shown in Theorem~\ref{thm:irr n-4}.

\begin{lem}\label{thm:flemma}
  Let $m\geq n>4$.  Then we have $f(m)+f(n)<f(m+2)+f(n-2)$, if $(m,n)=(6,6)$,
  and $f(m)+f(n)<f(m+1)+f(n-1)$, otherwise.
\end{lem}
\begin{proof}
If $(m,n)=(6,6)$, then it is clear. Otherwise,
\begin{align*}
f(m+1)+f(n-1)-f(m)-f(n) &= (m-3)-(n-4)+\delta_{m+1,6}+\delta_{n-1,6}
-\delta_{m,6}-\delta_{n,6}\\
&\geq (m-n)+1 -\delta_{m,6}-\delta_{n,6}\geq1.
\end{align*}
\end{proof}

\begin{lem}\label{thm:fflemma}
  Let $n,\ell$ be fixed integers. Let $n_1,\ldots,n_\ell$ be integers satisfying
  $n_i\geq 4$ for all $i$ and $\sum_{i=1}^\ell n_i -3(\ell-1)=n$.  Then
$$f(n_1)+\cdots+f(n_\ell)\leq \binom{n-2}{2}
-n\ell + \frac{\ell(\ell+3)}{2} +\delta_{\ell,n-5},$$ where the equality holds
if and only if for some $j$, $n_j=n-\ell+1$ and $n_i=4$ for $i\ne j$.
\end{lem}
\begin{proof}
The equality condition is straightforward to check.

Since the number of sequences $n_1,\ldots,n_\ell$ satisfying the conditions are
finite, there exists a sequence such that $f(n_1)+\cdots+f(n_\ell)$ is maximal.
Thus it is sufficient to show that if there are two integers $i,j$ with
$n_i\geq n_j>4$, then there is a sequence $n_1',n_2',\ldots,n_\ell'$
satisfying the conditions and $f(n_1)+\cdots+f(n_\ell)<
f(n_1')+\cdots+f(n_\ell')$. By Lemma~\ref{thm:flemma}, we can obtain such a
sequence by replacing $n_i$ and $n_j$ with $n_i+1$ and $n_j-1$ (or $n_i+2$ and
$n_j-2$ if $n_i=n_j=6$).
\end{proof}

The following is the main theorem of this section.

\begin{thm}\label{thm:n-4 ell}
  Let $P$ be a connected sum of $\ell$ irreducible polytopes.  Let $n$ be the
  number of vertices of $P$. If $\ell \leq n-4$, then
$$b_{n-4}(P) \leq \binom{n-3}{2} + \frac{\ell(\ell-3)}2
+\delta_{\ell,n-5},$$ where the equality holds if and only if $P$ is a
connected sum of a bipyramid and $\ell-1$ tetrahedrons.
\end{thm}
\begin{proof}
  Let $P=P_1 \# \cdots \# P_\ell$.  Let $P_i$ have $n_i$ vertices. We can assume
  $n_1\geq\cdots\geq n_\ell\geq 4$.  By Theorem~\ref{thm:irr n-4},
  Lemma~\ref{thm:mainlemma} and Lemma~\ref{thm:fflemma}, we have
\begin{align*}
  b_{n-4}(P) &= \sum_{i=1}^{\ell} b_{n_i -4}(P_i) + (n-3)(\ell-1)\\ &\leq
  \sum_{i=1}^{\ell} f(n_i)+(n-3)(\ell-1)\\
  &\leq \binom{n-2}{2} -n\ell
  + \frac{\ell(\ell+3)}{2} +\delta_{\ell,n-5} + (n-3)(\ell-1)\\
  &=\binom{n-3}{2} + \frac{\ell(\ell-3)}2 +\delta_{\ell,n-5}.
\end{align*}
The equality holds if and only if $P_1$ is a bipyramid and $P_i$ is the
tetrahedron for $i>1$.
\end{proof}

\section{Combinatorial rigidity of simplicial polytopes}
\label{sec:rigid}

Recall that two polytopes $P$ and $Q$ are combinatorially rigid if
$b_k(P)=b_k(Q)$ for all $k$ implies $P=Q$.
If $b_k(P)=b_k(Q)$ for all $k$, then the numbers of vertices of $P$ and $Q$ are the
same. In fact, the number of vertices of $P$ is determined by $b_2(P)$ as follows.

\begin{prop}\label{thm:b2}
Let $P$ be a polytope with $n$ vertices. Then
$$n=\frac{7+\sqrt{8 \cdot b_2(P)+1}}2.$$
\end{prop}
\begin{proof}
  Let $e$ (resp.~$f$) be the number of edges (resp.~faces) of $P$. Then by the
  Euler's theorem, we have $n-e+f=2$. Since $P$ is simplicial, we have
  $2e=3f$. Thus $e=3n-6$. Observe that $b_2(P)$ is the number of ways of
  choosing two vertices of $P$ which are not connected by an edge.  Thus
  $b_2(P)=\binom{n}{2}-e=\binom{n}{2}-3n+6$. Solving this equation, we get the
  theorem.
\end{proof}

\begin{prop}\label{thm:rigid prism}
A bipyramid and a semi-bipyramid are rigid.
\end{prop}
\begin{proof}
Let $P$ be a polytope with $b_k(P)=b_k(B_n)$ for all $n$. By
Proposition~\ref{thm:b2}, $P$ has $n$ vertices. Since
$b_{n-3}(P)=b_{n-3}(B_n)=0$, $P$ is irreducible. By Theorem~\ref{thm:irr n-4}
with $b_{n-4}(P)=b_{n-4}(B_n)$, we get $P=B_n$. Thus $B_n$ is rigid. We can
prove the rigidity of a semi-bipyramid in the same way using
  Theorem~\ref{thm:second max}.
\end{proof}

Theorem~\ref{thm:rigid} follows from the following propositions. Note that
$T_4\#T_4$ is rigid because it is the only polytope with $5$ vertices. Note also
that $O_6=B_6$. 

\begin{prop}\label{thm:tet}
  For $n\geq6$, the connected sum $T_4\# B_n$ is rigid.
\end{prop}
\begin{proof}
This is an immediate consequence of Theorem~\ref{thm:n-4 ell}.
\end{proof}

\begin{prop}\label{thm:cube}
  For $n\geq6$, the connected sum $O_6\# B_n$ is rigid.
\end{prop}
\begin{proof}
  Let $n'$ be the number of vertices of $O_6\# B_n$, i.e., $n'=n+3$.  Let $P$
  be a polytope with $b_k(P)=b_k(O_6\# B_n)$ for all $k$.  Since
  $b_{n'-3}(P)=b_{n'-3}(O_6\# B_n)=1$,
 we have $P\in\C(P_1\#P_2)$ for some
  irreducible polytopes $P_1$ and $P_2$. Let $n_1$ and $n_2$ be the number of
  vertices of $P_1$ and $P_2$ respectively.  Then $n_1+n_2=n+6$. Assume $n_1\geq
  n_2$. Since $b_{n'-4}(P)=b_{n'-4}(O_6\# B_n)$, by Lemma~\ref{thm:mainlemma},
  we have $b_{n_1-4}(P_1)+b_{n_2-4}(P_2)=b_{6-4}(O_6)+b_{n-4}(B_n)$. Since
  $b_{6-4}(O_6)+b_{n-4}(B_n) = f(6)+f(n)$ and $b_{n_1-4}(P_1)+b_{n_2-4}(P_2)
  \leq f(n_1)+f(n_2)$, by Lemma~\ref{thm:flemma}, we have $n_2\leq 6$.

  If $n_2=6$, then the equality holds and we get that $P_1=B_n$ and $P_2=O_6$.
  Since there is no irreducible simplicial polytope with $5$ vertices, we have
  $n_2\ne 5$.  Now assume $n_2=4$. Then $P_2=T_4$ and $n_1=n+2$.  Since
  $b_{n_1-4}(P_1)+b_{n_2-4}(P_2)=b_{6-4}(O_6)+b_{n-4}(B_n)=3+
  \binom{n-3}{2}-1+\delta_{n,6}$ and $b_{n_2-4}(P_2)=-1$, we get
\begin{equation}\label{eq:cube pf2}
b_{n_1-4}(P_1)=\binom{n-3}{2}+\delta_{n,6}+3.
\end{equation}

\begin{description}
\item[Case 1] $P_1$ is a bipyramid. Then
$$\binom{n-1}{2}-1+\delta_{n+2,6}=\binom{n-3}{2}+\delta_{n,6}+3.$$ Thus
we have $2n-5=4+\delta_{n,6}-\delta_{n,4}$, which does not have
an integer solution for $n\geq6$.
\item[Case 2] $P_1$ is not a bipyramid. Then by Theorem~\ref{thm:second
    max}, we get
$b_{n_1-4}(P_1)\leq \binom{n-3}2 + 2$, which is a contradiction to
\eqref{eq:cube pf2}.
\end{description}
Thus we can only have $P_1=B_n$ and $P_2=O_6$, thus $P=O_6\#B_n$ and $O_6\#B_n$
is rigid.
\end{proof}

\begin{prop}
The connected sum $T_4\# I_{12} $ is rigid.
\end{prop}
\begin{proof}
  Let $b_k(P)=b_k(T_4\#I_{12})$ for all $k$. Using the same argument in the
  proof of Proposition~\ref{thm:cube}, we have $P=P_1\# P_2$ for some
  irreducible polytopes $P_1$ and $P_2$ with $n_1$ and $n_2$ vertices
  respectively, and
  $b_{n_1-4}(P_1)+b_{n_2-4}(P_2)=b_{4-4}(T_4)+b_{12-4}(I_{12})=-1+0$.  Note that
  if a polytope $Q$ with $n$ vertices satisfies $b_{n-4}(Q)<0$, then
  $Q=T_4$. Moreover, $I_{12}$ is the only polytope with $12$ vertices without
  $4$-belts.  Thus $\{P_1,P_2\}=\{T_4,I_{12}\}$ and $P=T_4\# I_{12}$.
\end{proof}

\section{Further study}

In this paper, we show that for a $3$-dimensional irreducible simplicial
polytope $P$ with $n$ vertices, $b_{n-4}(P)$ has the maximum when $P=B_n$.

\begin{prob}
  Find the maximum of $b_k(P)$ for a $3$-dimensional irreducible simplicial
  polytope $P$ with $n$ vertices.
\end{prob}

Since the polytopes in Table~\ref{tab:rigid} are the only possible
$3$-dimensional reducible combinatorially rigid polytopes, the following problem
is solved if we prove the rigidity of those polytopes. Note that except
$I_{12}\# B_n$, the remaining polytopes are finite.

\begin{prob}
  Classify all $3$-dimensional reducible combinatorially rigid polytopes.
\end{prob}

\begin{prob}
  Generalize the arguments in this paper to arbitrary dimensional simplicial
  polytopes.
\end{prob}

\section*{acknowledgement}
The authors would like to thank Young Soo Kwon for suggesting us
Lemma~\ref{lem:face-transitive}.

\bibliographystyle{amsplain}

\end{document}